\documentclass[11pt]{amsart}

\usepackage[margin=1.1in]{geometry}
\usepackage{amsmath,amssymb,amsthm,url}

\title{Bounded gaps between primes with a given primitive root}

\author[P. Pollack]{Paul Pollack}
\address{Department of Mathematics\\University of Georgia\\Athens, GA 30602}
\email{pollack@uga.edu}

\DeclareMathAlphabet{\curly}{U}{rsfs}{m}{n}

\newtheorem{thm}{Theorem}[section]

\newtheorem*{apc}{Artin's primitive root conjecture}

\newtheorem{lem}[thm]{Lemma}

\newtheorem{prop}[thm]{Proposition}

\newtheorem*{theorem*}{Theorem}
\theoremstyle{remark}\newtheorem*{remark}{Remark}

\numberwithin{equation}{section}

\newcommand{\diff}{\mathrm{d}}
\newcommand{\leg}[2]{\genfrac{(}{)}{}{}{#1}{#2}}
\newcommand\Li{\mathrm{Li}}
\newcommand\I{\mathcal{I}}
\newcommand\F{\mathbf{F}}

\newcommand\Q{\mathbf{Q}}
\newcommand\R{\mathbf{R}}
\newcommand\Hh{\mathcal{H}}
\newcommand\lcm{\mathrm{lcm}}
\newcommand\Gal{\mathrm{Gal}}
\newcommand\A{\curly{A}}
\newcommand\Pp{\curly{P}}
\newcommand\Qq{\curly{Q}}
\newcommand\Frob{\mathrm{Frob}}
\newcommand\E{\mathbf{E}}
\newcommand\Prob{\mathbf{Pr}}

\renewcommand{\phi}{\varphi}
\makeatletter
\renewcommand{\pod}[1]{\mathchoice
  {\allowbreak \if@display \mkern 18mu\else \mkern 8mu\fi (#1)}
  {\allowbreak \if@display \mkern 18mu\else \mkern 8mu\fi (#1)}
  {\mkern4mu(#1)}
  {\mkern4mu(#1)}
}

\begin{document}

\begin{abstract} Fix an integer $g \neq  -1$ that is not a perfect square. In 1927, Artin conjectured that there are infinitely many primes for which $g$ is a primitive root. Forty years later, Hooley showed that Artin's conjecture follows from the Generalized Riemann Hypothesis (GRH). We inject Hooley's analysis into the Maynard--Tao work on bounded gaps between primes. This leads to the following GRH-conditional result: \emph{Fix an integer $m \geq 2$. If $q_1 < q_2 < q_3 < \dots$ is the sequence of primes possessing $g$ as a primitive root, then $\liminf_{n\to\infty} (q_{n+(m-1)}-q_n) \leq C_m$, where $C_m$ is a finite constant that depends on $m$ but not on $g$.} We also show that the primes $q_n, q_{n+1}, \dots, q_{n+m-1}$ in this result may be taken to be consecutive.
\end{abstract}

\maketitle

\section{Introduction}
The following conjecture was proposed by Emil Artin in the course of a September 1927 conversation with Helmut Hasse:

\begin{apc} Fix an integer $g \neq  -1$ that is not a square. There are infinitely many primes $p$ for which $g$ is a primitive root modulo $p$. In fact, the number of such $p\leq x$ is (as $x\to\infty$) asymptotically $c_g \pi(x)$ for a certain $c_g > 0$.
\end{apc}

While there is a substantial literature surrounding Artin's conjecture (lovingly catalogued in the survey \cite{moree12}), we still know infuriatingly little. In particular, there is no specific value of $g$ which is known to occur as a primitive root for infinitely many primes. However, thanks to work of Heath-Brown \cite{HB86} (refining earlier results of Gupta and Murty \cite{GM84}), we know that at least one of $2, 3$, and $5$ has this property. In fact, one can replace ``$2, 3$, and $5$'' with any list of three nonzero multiplicatively independent integers.

In a seminal 1967 paper, Hooley \cite{hooley67} (see also his exposition in \cite[Chapter 3]{hooley76}) showed that the Chebotarev density theorem with a sufficiently sharp error term would imply the quantitative form of Artin's conjecture. Moreover, he showed that such a variant of Chebotarev's density theorem --- at least for the cases relevant for this application --- follows from the Generalized Riemann Hypothesis (GRH) for Dedekind zeta functions. Thus, under GRH, we have a fairly satisfactory complete solution to Artin's conjecture.

In this paper, we combine Hooley's work on Artin's conjecture with recent methods used to study gaps between primes. In sensational work of Maynard \cite{maynard14} and Tao, it is shown that $\liminf_{n\to\infty} (p_{n+m-1} - p_n) < \infty$ for every $m$.
Here $p_1 < p_2 < p_3 < \dots$ is the sequence of all primes, in the usual order. Our main theorem is an analogous bounded gaps result for primes possessing a prescribed primitive root.
\begin{thm}[conditional on GRH]\label{thm:main} Fix an integer $g \neq -1$ and not a square. Let $q_1 < q_2 < q_3 < \dots$ denote the sequence of primes for which $g$ is a primitive root. Then for each $m$,
\[ \liminf_{n\to\infty} (q_{n+m-1} - q_n) \leq C_m, \]
where $C_m$ is a finite constant depending on $m$ but not on $g$.
\end{thm}

In the concluding section of the paper, we show how to modify the proof of Theorem \ref{thm:main} to impose the additional restriction that the $m$ primes $q_n, q_{n+1}, \dots, q_{n+m-1}$ are in fact \emph{consecutive} (Theorem \ref{thm:consecutive}).

We remark that other recent work producing bounded gaps between primes in special sets has been done by Thorner \cite{thorner14}, who handles primes restricted by  Chebotarev conditions, and by Li and Pan \cite{LP14}, who work with primes $p$ for which $p+2$ is an `almost prime'.

\subsection*{Notation} The letters $p$ and $q$ always denote primes. Implied constants may depend on $k$ and on $g$, unless otherwise noted.

\section{Technical preparation}
\subsection{Configurations of quadratic residues and nonresidues} We will use that certain configurations of residues and nonresidues are guaranteed to appear for all large enough primes. This is a fairly standard consequence of the Riemann Hypothesis for curves as proved by Weil, but we give the argument for completeness. The following lemma is a special case of \cite[Corollary 2.3]{wan97}.

\begin{lem}\label{lem:weil} Let $p$ be a prime. Suppose that $f(T)$ is a monic polynomial in $\F_p[T]$ of degree $d$ and that $f(T)$ is not a square in $\F_p[T]$. Then
\[ \left|\sum_{a \bmod{p}} \leg{f(a)}{p}\right| \leq (d-1)\sqrt{p}. \]
\end{lem}

\begin{lem}\label{lem:quadres} Let $p$ be a prime, and let $k$ be a positive integer. Suppose that $h_1, \dots, h_k$ are integers no two of which are congruent modulo $p$. Suppose $\epsilon_1, \dots, \epsilon_k \in \{\pm 1\}$. The number of mod $p$ solutions $n$ to the system of equations
\begin{equation}\label{eq:systemstar} 	\leg{n+h_i}{p}= \epsilon_i \quad\text{for all}\quad 1 \leq i \leq k \end{equation}
is at least $\frac{p}{2^k} - (k-1)\sqrt{p} - k$.
\end{lem}
\begin{proof} For each $n$, let $\iota(n) = \frac{1}{2^k} \prod_{i=1}^{k} (1+\epsilon_i \leg{n+h_i}{p})$.  If we suppose $n \not\equiv -h_1$, \dots, $-h_k \pmod{p}$, then $\iota(n) = 1$ when \eqref{eq:systemstar} holds and $=0$ otherwise. Since $|\iota(n)| \leq 1$ for all $n$, the number of solutions to \eqref{eq:systemstar} is at least
$-k + \sum_{n \bmod{p}} \iota(n)$.  For each subset $S \subset \{1, 2, 3, \dots, k\}$, put $f_S(T) = \prod_{i\in S} (T+h_i) \in \F_p[T]$. Then
\[ \sum_{n \bmod{p}} \iota(n) = \frac{1}{2^k}\sum_{S \subset \{1, 2, \dots, k\}} \left(\prod_{i \in S}\epsilon_i\right) \sum_{n \bmod{p}} \leg{f_S(n)}{p}. \]
If $S=\emptyset$, then $f_S=1$, and we get a contribution of $\frac{p}{2^k}$. In all other cases,
$f_S$ is a nonsquare polynomial of degree at most $k$. By Lemma \ref{lem:weil}, the total contribution from all nonempty subsets of $\{1, 2, \dots, k\}$ is bounded in absolute value by
$\frac{2^k-1}{2^k} (k-1)\sqrt{p} \le (k-1)\sqrt{p}$. Thus, $\sum_{n \bmod{p}} \iota(n) \ge \frac{p}{2^k} - (k-1)\sqrt{p}$, and the lemma follows.
\end{proof}

\subsection{Effective Chebotarev} The next result is due in essence to Lagarias and Odlyzko \cite{LO77}, although the precise formulation we give is due to Serre \cite[\S2.4]{serre81}:

\begin{thm}[conditional on GRH]\label{thm:serre}  Let $L$ be a finite Galois extension of $\Q$ with Galois group $G$, and let $C$ be a conjugacy class of $G$. The number of unramified primes $p \leq x$ whose Frobenius conjugacy class $(p, L/\Q)=C$ is given by
	\[ \frac{\#C}{\#G}\Li(x)+ O\left(\frac{\#C}{\#G}x^{1/2}(\log|\Delta_L| + [L:\Q]\log{x})\right), \]
for all $x\geq 2$. Here $\Delta_L$ denotes the discriminant of $L$ and the $O$-constant is absolute.
\end{thm}

To apply Theorem \ref{thm:serre}, we require an upper bound for the term $\log|\Delta_L|$. The following result, which is contained in \cite[Proposition 6]{serre81}, suffices for our applications.

\begin{lem}\label{lem:discbound} For every Galois extension $L/\Q$, we have
\[ \log|\Delta_L| \leq ([L:\Q]-1) \sum_{p \mid \Delta_L} \log{p} + [L:\Q] \log [L:\Q]. \]
\end{lem}

\section{Proof of Theorem \ref{thm:main}}\label{sec:proof}
\subsection{The Maynard--Tao strategy} We begin by recalling the strategy of \cite{maynard14} for producing bounded gaps between primes. Let $k \geq 2$ be a fixed positive integer, and let $\Hh = \{h_1 < h_2 < \dots < h_k\}$ denote a fixed \emph{admissible $k$-tuple}, i.e., a set of $k$ distinct integers that does not occupy all of the residue classes modulo $p$ for any prime $p$. With $N$ a large positive integer, we seek values of $n$ belonging to the dyadic interval $[N, 2N)$ for which the shifted tuple $n+h_1, n+h_2, \dots, n+h_k$ contains several primes.

Let $W := \prod_{p \leq \log\log\log{N}} p$. Choose an integer $\nu$ so that $\gcd(\nu+h_i,W)=1$ for all $1 \leq i\leq k$; the existence of such a $\nu$ is implied by the admissibility of $\Hh$. We restrict attention to integers $n \equiv \nu\pmod{W}$. This has the effect of pre-sieving the values of $n$ to ensure that none of the $n+h_i$ have any small prime factors. Let $w(n)$ denote nonnegative weights (to be chosen momentarily), and let $\chi_{\Pp}$ denote the characteristic function of the set $\Pp$ of prime numbers. One studies the sums
\[ S_1:= \sum_{\substack{N \leq n < 2N \\ n \equiv \nu\pmod{W}}} w(n)\quad\text{and}\quad S_2:= \sum_{\substack{N \leq n < 2N \\ n\equiv \nu\pmod{W}}} \left(\sum_{i=1}^{k} \chi_{\Pp}(n+h_i)\right)w(n). \]
The ratio $S_2/S_1$ is a weighted average of the number of primes among $n+h_1, \dots, n+h_k$, as $n$ ranges over $[N,2N)$. Consequently, if $S_2 > (m-1) S_1$ for the positive integer $m$, then at least $m$ of the numbers $n+h_1, \dots, n+h_k$ are primes. So if the inequality $S_2 > (m-1) S_1$ is achieved for a sequence of $n$ tending to infinity, then $\liminf (p_{n+m-1} - p_n) \leq h_k-h_1 < \infty$.

As we have described it so far, this strategy goes back to Goldston--Pintz--Y{\i}ld{\i}r{\i}m. The key innovation in the approach of Maynard--Tao is the choice of congenial weights $w(n)$. The following result, which is a restatement of \cite[Proposition 4.1]{maynard14}, is crucial.

\begin{prop}\label{prop:main-maynard} Let $\theta$ be a positive real number with $\theta < \frac{1}{4}$.
Let $F$ be a piecewise differentable function supported on the simplex $\{(x_1, \dots, x_k): \text{each }x_i \geq 0, \sum_{i=1}^{k} x_i \leq 1\}$. With $R:= N^{\theta}$, put
\[ \lambda_{d_1, \dots, d_k}:= \left(\prod_{i=1}^{k} \mu(d_i) d_i\right) \sum_{\substack{r_1, \dots, r_k \\ d_i \mid r_i\,\forall i \\ (r_i, W)=1\,\forall i}} \frac{\mu(\prod_{i=1}^{k} r_i)^2}{\prod_{i=1}^{k} \phi(r_i)} F\left(\frac{\log{r_1}}{\log{R}},  \dots, \frac{\log{r_k}}{\log{R}}\right)\]
whenever $\gcd(\prod_{i=1}^{k} d_i,W)=1$, and let $\lambda_{d_1, \dots, d_k}=0$ otherwise. Let
\[ w(n):= \left(\sum_{d_i \mid n+h_i\,\forall i} \lambda_{d_1, \dots, d_k}\right)^2. \]
Then as $N\to\infty$,
\begin{align*} S_1 &\sim \frac{\phi(W)^k}{W^{k+1}} N (\log{R})^k I_k(F), ~~\text{and} \\
S_2 &\sim \frac{\phi(W)^k}{W^{k+1}} \frac{N}{\log{N}} (\log{R})^{k+1} \sum_{m=1}^{k} J_k^{(m)}(F), \end{align*}
provided that $I_k(F) \neq 0$ and $J_k^{(m)}(F) \neq 0$ for each $m$, where
\begin{align*} I_k(F) :&= \idotsint_{[0,1]^{k}} F(t_1, \dots, t_k)^2\, \diff t_1 \diff t_2 \cdots \diff t_k, \\
	J_k^{(m)}(F):&= \idotsint_{[0,1]^{k-1}} \left(\int_{0}^{1} F(t_1, \dots, t_k)\, \diff t_m\right)^2 \diff t_1 \cdots  \diff t_{m-1} \diff t_{m+1} \cdots \diff t_k.
\end{align*}
\end{prop}

 From our interpretation of $S_2/S_1$ as a weighted average, we know that there is an $n \in [N,2N)$ for which at least $S_2/S_1$ of the numbers $n+h_1, \dots, n+h_k$ are prime. Proposition \ref{prop:main-maynard} shows that $S_2/S_1 \to \theta \frac{\sum_{m=1}^{k} J_k^{(m)}(F)}{I_k(F)}$, as $N\to\infty$. Let
\begin{equation}\label{eq:mkdef}M_k := \sup_{F} \frac{\sum_{m=1}^{k} J_k^{(m)}(F)}{I_k(F)},\end{equation} where the supremum is taken over all $F$ satisfying the previously indicated conditions. Upon choosing $\theta$ close to $\frac{1}{4}$, and $F$ so that the supremum appearing in the definition \eqref{eq:mkdef} is close to $M_k$, we find that infinitely often, at least $\lceil \frac{1}{4}M_k\rceil$ of the numbers $n+h_1,\dots, n+h_{k}$ are prime. The following lower bound on $M_k$ is due to Maynard \cite[Proposition 4.3]{maynard14}.

\begin{prop}\label{prop:mklower} $M_k\to\infty$ as $k\to\infty$. In fact, for all sufficiently large values of $k$,
	\[ M_k > \log{k}-2\log\log{k}-2.\]
\end{prop}

Consequently, once $k$ is a little larger than $e^{4m}$, we have $\lceil \frac{1}{4} M_k\rceil > m-1$. From the above discussion, $\liminf_{n\to\infty} (p_{n+m-1}-p_n) \leq h_k-h_1 < \infty$ for every admissible $k$-tuple $\Hh$. Choosing $\Hh$ carefully, this argument gives $\liminf_{n\to\infty} (p_{n+m-1}-p_n) \ll m^3 e^{4m}$; see the proof of \cite[Theorem 1.1]{maynard14} for details.

\subsection{Modifying Maynard--Tao} For the rest of the paper, we fix an integer $g \neq -1$ that is not a square. Let $\tilde{\Pp}$ denote the set of primes having $g$ as a primitive root.
Fix an integer $k \geq 2$, and let
\[ K:= 9k^2 \cdot 4^k. \]
We let $\Hh$ denote the admissible $k$-tuple with $h_i = (i-1)K!$ for all $1 \leq i\leq k$; that is,
\begin{equation}\label{eq:hidef} \Hh:= \{0, K!, 2K!, \dots, (k-1)K!\}. \end{equation}
In what follows, we think of $N$ as very large, in particular much larger than $g$. We use the Maynard--Tao strategy to detect $n \in [N,2N)$ for which the list $n+h_1, \dots, n+h_k$ contains several primes belonging to $\tilde{\Pp}$. Let $g_0$ denote the discriminant of the quadratic field $\Q(\sqrt{g})$. Set
\[ W := \lcm[g_0, \prod_{p \leq \log\log\log{N}} p]. \]
Once again, we pre-sieve values of $n$ by putting $n$ in an appropriate residue class $\nu \bmod W$. Whereas
Maynard could choose any $\nu$ with $\gcd(\nu+h_i,W)=1$ for all $1 \leq i \leq k$, we must tread more carefully. We choose $\nu$ so that the primes detected by the sieve are heavily biased towards having $g$ as a primitive root.

\begin{lem}\label{lem:presieving} We can choose an integer $\nu$ with all of the following properties:
\begin{enumerate}
\item[(i)] $\nu+h_i$ is coprime to $W$ for all $1\leq i \leq k$,
\item[(ii)] $\nu+h_i-1$ is coprime to $\prod_{2 < p \leq \log\log\log{N}} p$ for all $1\leq i \leq k$,
\item[(iii)] The Kronecker symbol $\leg{g_0}{\nu+h_i} = -1$ for all $1 \leq i \leq k$.
\end{enumerate}
\end{lem}
\begin{proof} Factor $g_0$ as a product $D_1 D_2 \dots D_\ell$ of coprime prime discriminants, where the \emph{prime discriminants} are the numbers $-4, -8, 8$, and $(-1)^{\frac{p-1}{2}} p$ for odd primes $p$. Reordering the factorization if necessary, we can assume all of the following:
\begin{itemize}
\item If all $|D_i| \leq K$ and $g_0$ is even, then $D_1 \in \{-4, 8, 8\}$.
\item If all $|D_i| \leq K$, $g_0$ is odd, and $\ell > 1$, then $|D_1| \geq 5$.
\item If some $|D_i| > K$, then $|D_1| > K$.
\end{itemize}
We begin by choosing any odd integer $\nu_1$ that avoids the residue classes $-h_1, \dots, -h_k$, $1-h_1$, $\dots, 1-h_k$ modulo $p$ for each odd prime $p \leq \log\log\log{N}$ not dividing $D_1$. Note that when $p \leq K$, the only requirement on $\nu_1$ is that it avoids the residue classes $0$ and $1$ mod $p$, while when $p > K$, we are to avoid at most $2k$ of the $p > K > 2k$ residue classes modulo $p$. So such a choice of $\nu_1$ certainly exists by the Chinese remainder theorem. We choose $\nu$ to satisfy
\[ \nu \equiv \nu_1 \pmod{[W/D_1, 2]}. \]
To ensure (i), (ii), and (iii), it suffices to impose a further condition on $\nu$ guaranteeing \
\begin{enumerate}
\item[(i$'$)] $\nu+h_i$ is coprime to all odd $p$ dividing $D_1$ for all $1 \leq i \leq k$,
\item[(ii$'$)] $\nu+h_i-1$ is coprime to all odd $p$ dividing $D_1$ for all $1 \leq i \leq k$,
\item[(iii$'$)] $\leg{D_1}{\nu+h_i} = -\leg{D_2 \cdots D_l}{\nu_1+h_i}$ for all $1 \leq i \leq k$.
\end{enumerate}
Notice that for all $1 \leq i \leq k$, we have $\leg{D_2 \cdots D_l}{\nu_1+h_i} \neq 0$ by the choice of $\nu_1$.

\subsubsection*{Case I: All $|D_i| \leq K$.} In this case, (i$'$) and (ii$'$) are satisfied as long as $\nu \not\equiv 0\text{ or }1\pmod{p}$ for any odd $p$ dividing $D_1$, while (iii$'$) is satisfied as long as
\[ \leg{D_1}{\nu} = -\leg{D_2 \cdots D_l}{\nu_1}. \]

Assume first that $g_0$ is even. Then $D_1 \in \{-4, -8, 8\}$ and (i$'$) and (ii$'$) hold vacuously. Choose $\nu_2$ so that $\leg{D_1}{\nu_2} = -\leg{D_2 \cdots D_l}{\nu_1}$. We ensure (iii$'$) by selecting $\nu$ as any solution to the simultaneous congruences
\begin{equation}\label{eq:simultaneous} \nu \equiv \nu_1 \pmod{[W/D_1, 2]} \quad\text{and}\quad \nu \equiv \nu_2 \pmod{D_1}.\end{equation}
While the moduli here share a factor of $2$, it is clear that these congruences still admit a simultaneous solution, since the only $2$-adic information encoded by the first congruence is that $\nu$ is odd, which is certainly compatible with the second!

Now assume instead that $g_0$ is odd, so that $|D_1|$ is an odd prime. Either $|D_1|=3$ and $\ell=1$, or $|D_1| \geq 5$. If the former, then (i$'$), (ii$'$), and (iii$'$) hold upon selecting $\nu_2=2$ and choosing $\nu$ to satisfy \eqref{eq:simultaneous}. If the latter, choose $\nu_2 \not\equiv 1\pmod{D_1}$ with
$\leg{D_1}{\nu_2} = -\leg{D_2 \cdots D_l}{\nu_1}$; this is possible since that equality of Legendre symbols holds for a total of $\frac{|D_1|-1}{2} > 1$ residue classes $\nu_2 \bmod{D_1}$. Once again, choosing $\nu$ to satisfy \eqref{eq:simultaneous} completes the proof.

\subsubsection*{Case II: Some $|D_i| > K$.} In this case, $|D_1| > K$. Since $K >8$, we see that $|D_1|$ is an odd prime. To satisfy (i$'$), (ii$'$), and (iii$'$), it suffices to show that there is an integer $\nu_2 \not\equiv  1-h_1, \dots, 1-h_k \pmod{D_1}$ with
\begin{equation}\label{eq:quadressystem} \leg{\nu_2+h_i}{|D_1|} = -\leg{D_2 \cdots D_l}{\nu_1+h_i} \quad\text{for all $1 \leq i \leq k$}, \end{equation}
for in that case we can choose $\nu$ as any solution to \eqref{eq:simultaneous}. (We used here that $\leg{D_1}{\nu+h_i}=\leg{\nu+h_i}{|D_1|}$.) The integers $h_1, \dots, h_k$ are incongruent modulo $D_1$, as each nonzero difference $h_j - h_i = (j-i)K!$ has only prime factors smaller than $K$. So Lemma \ref{lem:quadres} gives that the number of  $\nu_2\bmod{D_1}$ satisfying \eqref{eq:quadressystem} is at least $|D_1|/2^k - (k-1)\sqrt{|D_1|} - k$. Since $|D_1| > K = 9k^2 \cdot 4^k$, this count of solutions exceeds $k$. In particular, we can satisfy \eqref{eq:quadressystem} with $\nu_2 \not\equiv  1-h_1, \dots, 1-h_k \pmod{D_1}$.
\end{proof}
	
Assume that $\nu$ has been chosen to to satisfy the conditions of Lemma \ref{lem:presieving}. We let $R=N^{\theta}$, with $\theta$ to be specified momentarily, and we define the weights $w(n)$ exactly as in the statement of Proposition \ref{prop:main-maynard}. We let \[ \tilde{S}_{1}:= \sum_{\substack{N \leq n < 2N \\ n \equiv \nu\pmod{W}}} w(n)\quad\text{and}\quad \tilde{S}_{2}:= \sum_{\substack{N \leq n < 2N \\ n\equiv \nu\pmod{W}}} \left(\sum_{i=1}^{k} \chi_{\tilde{\Pp}}(n+h_i)\right)w(n). \] Theorem \ref{thm:main} is a consequence of the following result, established in the next section.

\begin{prop}[assuming GRH]\label{prop:main} Fix a positive real number $\theta < \frac{1}{4}$. As $N\to\infty$, we have the same asymptotic estimates for $\tilde{S}_1$ and $\tilde{S}_2$ as those for $S_1$ and $S_2$ given in Proposition \ref{prop:main-maynard}.
\end{prop}	

Once Proposition \ref{prop:main} has been established, the earlier analysis we applied to Maynard's Proposition \ref{prop:main-maynard} applies, and we immediately obtain Theorem \ref{thm:main}.

\subsection{Proof of Proposition \ref{prop:main}} The $\tilde{S}_1$ estimate is established in precisely the same way as Maynard's $S_1$ estimate in Proposition \ref{prop:main-maynard}; see the proofs of Lemmas 5.1 and 6.2 in \cite{maynard14}. So we describe only the estimation of $\tilde{S}_2$. We write $\tilde{S}_2 = \sum_{m=1}^{k} \tilde{S}_2^{(m)}$, where each
\[ \tilde{S}_2^{(m)} := \sum_{\substack{N \leq n < 2N \\ n\equiv \nu\pmod{W}}} \chi_{\tilde{\Pp}}(n+h_m) w(n). \]
This is precisely analogous to Maynard's decomposition of $S_2$ as $\sum_{m=1}^{k} S_2^{(m)}$, where $S_2^{(m)}:= \sum_{\substack{N \leq n < 2N \\ n\equiv \nu\pmod{W}}} \chi_{\Pp}(n+h_m) w(n)$. Maynard's proof of Proposition \ref{prop:main-maynard} gives that each
\[ S_2^{(m)} \sim \frac{\phi(W)^k}{W^{k+1}} \frac{N}{\log{N}} (\log{R})^{k+1} \cdot J_k^{(m)}(F). \]
So to prove Proposition \ref{prop:main}, it suffices to show that for each $m$, we have
\begin{equation}\label{eq:os2} S_2^{(m)} - \tilde{S}_2^{(m)} = o\left(\frac{\phi(W)^k}{W^{k+1}} N (\log{N})^{k}\right), \end{equation}
as $N\to\infty$. From now on, we think of $m$ as fixed, and we focus our energies on proving \eqref{eq:os2}.

To prepare for the proof of \eqref{eq:os2}, for each prime $q$, we let $\Pp_q^{(0)}$ denote the set of all primes $p$ satisfying
\begin{equation}\label{eq:qtestfail} p\equiv 1\pmod{q}\quad\text{and}\quad g^{\frac{p-1}{q}} \equiv 1\pmod{p}. \end{equation}
Let
\[ \Pp_q := \Pp_q^{(0)} \setminus \bigcup_{q' < q}\Pp_{q'}^{(0)}. \]
Provided that the argument is not a prime divisor of $g$,
\begin{equation}\label{eq:setinclusions} 0 \leq \chi_{\Pp}- \chi_{\tilde{\Pp}} \leq \sum_{q} \chi_{\Pp_{q}}. \end{equation}
Indeed, if $p$ is a prime not dividing $g$, then either $g$ is a primitive root mod $p$ or $g$ is a $q$th power residue mod $p$ for some prime $q$ dividing $p-1$. From \eqref{eq:setinclusions}, it follows immediately that
\begin{equation}\label{eq:lessthansum} 0 \leq S_2^{(m)} - \tilde{S}_2^{(m)} \leq \sum_{q} \sum_{\substack{N\leq n < 2N \\ n \equiv \nu \pmod{W}}}\chi_{\Pp_{q}}(n+h_m) w(n). \end{equation}

We claim that the primes $q \leq \log\log\log{N}$ make no contribution to the right-hand side of \eqref{eq:lessthansum}. Indeed, suppose $p:=n+h_m$ is prime with $N \leq n < 2N$ and $n\equiv \nu\pmod{W}$. By Lemma \ref{lem:presieving}(ii), the number $p-1$ has no odd prime factors up to $\log\log\log{N}$; it follows trivially that $\chi_{\Pp_q}(p)=0$ for odd $q \leq \log\log\log{N}$. By Lemma \ref{lem:presieving}(iii), $\chi_{\Pp_2}(p) = 0$, since modulo $p$,
\[ g^{\frac{p-1}{2}} \equiv \leg{g}{p} = \leg{g}{n+h_m} = \leg{g_0}{n+h_m} = -1. \]
Thus, the right-hand side of \eqref{eq:lessthansum} can be rewritten as $\sideset{}{_1}\sum + \sideset{}{_2}\sum + \sideset{}{_3}\sum + \sideset{}{_4}\sum$, where the subscripts correspond to the following ranges of $q$:
\begin{enumerate}
\item[(1)] $\log\log\log{N} < q \leq (\log{N})^{100k}$,
\item[(2)] $(\log{N})^{100k} < q \leq N^{1/2} (\log{N})^{-100k}$,
\item[(3)] $N^{1/2}(\log{N})^{-100k} < q \leq N^{1/2} (\log{N})^{100k}$,
\item[(4)] $q> N^{1/2} (\log{N})^{100k}$.
\end{enumerate}
We treat all four ranges of $q$ separately.

\subsubsection{Estimation of $\sideset{}{_2}\sum$ and $\sideset{}{_4}\sum$} We need the following lemma, which facilitates later applications of Cauchy--Schwarz.

\begin{lem}\label{lem:csprep} We have
	\[ \sum_{\substack{N \leq n < 2N \\ n \equiv \nu\pmod{W}}} w(n)^2 \ll F_{\max}^4 \frac{N}{W} (\log{R})^{19k}.  \]
\end{lem}
\begin{proof} Let $\mathbf{d}= (d_1, \dots, d_k)$, $\textbf{e}=(e_1, \dots, e_k)$, $\textbf{f}= (f_1, \dots, f_k)$, and $\textbf{g} = (g_1, \dots, g_k)$ represent $k$-tuples of positive integers. Expanding the sum using the definition of $w(n)$ gives
\[ \sum_{\substack{N \leq n < 2N \\ n \equiv \nu\pmod{W}}} \sum_{\substack{\mathbf{d}, \mathbf{e}, \mathbf{f}, \mathbf{g} \\ [d_i, e_i, f_i, g_i] \mid n+h_i\,\forall i}}  \lambda_{\mathbf{d}} \lambda_{\mathbf{e}} \lambda_{\mathbf{f}} \lambda_{\mathbf{g}} = \sum_{\mathbf{d}, \mathbf{e}, \mathbf{f}, \mathbf{g}} \lambda_{\mathbf{d}} \lambda_{\mathbf{e}} \lambda_{\mathbf{f}} \lambda_{\mathbf{g}} \sum_{\substack{N \leq n < 2N \\ n \equiv \nu\pmod{W} \\  [d_i, e_i, f_i, g_i] \mid n+h_i\,\forall i}}  1. \]
Remembering that $\lambda_{d_1, \dots, d_k}$ vanishes unless $d_1 \cdots d_k$ is prime to $W$, we see that a quadruple $\textbf{d}, \textbf{e}, \textbf{f}, \textbf{g}$ makes no contribution to the right-hand side unless the numbers $[d_i, e_i, f_i, g_i]$, for $1 \leq i \leq k$, are pairwise coprime and all coprime to $W$. In that case, the conditions on $n$ in the inner sum put $n$ in a uniquely determined congruence class modulo $W \prod_{i=1}^{k} [d_i, e_i, f_i, g_i]$. It follows that our sum is bounded above by
\[ \sum_{\mathbf{d}, \mathbf{e}, \mathbf{f}, \mathbf{g}} |\lambda_{\mathbf{d}} \lambda_{\mathbf{e}} \lambda_{\mathbf{f}} \lambda_{\mathbf{g}}| \left(\frac{N}{W \prod_{i=1}^{k} [d_i, e_i, f_i, g_i]} + 1\right).\]
Let \begin{equation}\label{eq:rdef}r := \prod_{i=1}^{k} [d_i, e_i, f_i, g_i].\end{equation} Since $\lambda_{d_1, \dots, d_k}$ vanishes unless $d_1 \cdots d_k$ is a squarefree integer smaller than $R$, we may restrict attention to squarefree $r < R^4$. Given $r$, there are $\tau_{15k}(r)$ choices of $\textbf{d}, \textbf{e}, \textbf{f}$, and $\textbf{g}$ giving \eqref{eq:rdef}. Hence, writing $\lambda_{\max} = \max_{d_1, \dots, d_k} |\lambda_{d_1, \dots, d_k}|$, we find that
\begin{align} \sum_{\mathbf{d}, \mathbf{e}, \mathbf{f}, \mathbf{g}} |\lambda_{\mathbf{d}} \lambda_{\mathbf{e}} \lambda_{\mathbf{f}} \lambda_{\mathbf{g}}| \bigg(\frac{N}{W \prod_{i=1}^{k} [d_i, e_i, f_i, g_i]} + 1\bigg)&\leq \lambda_{\max}^4 \sum_{r < R^4} \mu^2(r)\tau_{15k}(r)\left(\frac{N}{Wr}+1\right)  \notag\\
	\label{eq:almostdone}&\leq \lambda_{\max}^4 \left(\frac{N}{W} + R^4\right) \sum_{r < R^4} \frac{\mu^2(r)\tau_{15k}(r)}{r}.
\end{align}
The remaining sum on $r$ is bounded above by $\prod_{p < R^4}(1+15k/p) \ll (\log{R})^{15k}$. Since $R= N^{\theta}$ with $\theta< \frac{1}{4}$ fixed, we get that $R^4 \ll N/W$. Finally, we recall that $\lambda_{\max} \ll F_{\max} (\log{R})^k$ (see \cite[eqs. (5.9) and (6.3)]{maynard14}). Inserting these estimates into \eqref{eq:almostdone} gives the lemma.
\end{proof}

\begin{proof}[Proof that $\sideset{}{_2}\sum = o\left(\frac{\phi(W)^k}{W^{k+1}} N (\log{N})^{k}\right)$] Let $\Qq$ be the union of the sets $\Pp_q$ for $(\log{N})^{100k} < q \leq N^{1/2} (\log{N})^{-100k}$. Then $\sideset{}{_2}\sum = \displaystyle\sum\nolimits_{\substack{N \leq n < 2N \\ n \equiv \nu \pmod{W}}} \chi_{\Qq}(n+h_m) w(n)$. Applying Cauchy--Schwarz and Lemma \ref{lem:csprep}, we see that
\begin{equation}\label{eq:crudesum4-0} \sideset{}{_2}\sum \ll F_{\max}^2 W^{-1/2} N^{1/2} (\log{R})^{9.5k} \bigg(\sum_{\substack{N \leq n < 2N \\ n \equiv \nu \pmod{W}}} \chi_{\Qq}(n+h_m)\bigg)^{1/2}.\end{equation}
The remaining sum on $n$ is certainly bounded above by the total number of primes $p \in [N,3N]$ belonging to $\Qq$. For each such $p$, we may select a $q$ with $(\log{N})^{100k} < q \leq N^{1/2} (\log{N})^{-100k}$ for which \eqref{eq:qtestfail} holds. Given $q$, we count the number of corresponding $p$ using effective Chebotarev.

Since $g$ is fixed and $q$ is large, we see that $g \not\in (\Q^{\times})^{q}$. So by a theorem of Capelli on irreducible binomials, the extension $\Q(\sqrt[q]{g})/\Q$ has degree $q$. For later use, we note that the discriminant of $\Q(\sqrt[q]{g})$ divides $(gq)^{q}$, and so the only ramified primes divide $gq$. By a theorem of Dedekind--Kummer, a prime $p \in [N,3N]$ satisfies \eqref{eq:qtestfail} precisely when $p$ splits completely in $L:=\Q(\zeta_q, \sqrt[q]{g})$. To continue, we need to know the degree of $L/\Q$. Now $\sqrt[q]{g}$ is not contained in $\Q(\zeta_q)$ --- otherwise, $\sqrt[q]{g}$ would generate a Galois extension of $\Q$, contradicting that $\Q(\sqrt[q]{g})$ contains only a single $q$th root of unity (since it can be viewed as a subfield of $\R$). So by another application of Capelli's theorem,
\[ [L:\Q] = [L:\Q(\zeta_q)]\cdot [\Q(\zeta_q): \Q] = q(q-1). \]
Moreover, since $q$ is the only ramified prime in $\Q(\zeta_q)/\Q$, the only primes that may ramify in $L/\Q$ all divide $gq$. By Lemma \ref{lem:discbound}, 
\begin{align*} \log|\Delta_L| \ll q^2\log{(|g|q)} \ll q^2 \log{N}. \end{align*}
We plug this estimate into Theorem \ref{thm:serre}, taking $C$ as the conjugacy class of the identity. We find that the number of $p \in [N,3N]$ for which \eqref{eq:qtestfail} holds for a given $q$ is
\[ \frac{1}{q(q-1)} \int_{N}^{3N} \frac{dt}{\log{t}} + O(N^{1/2}\log{N}). \]
Summing this upper bound over primes $q$ with $(\log{N})^{100k} < q \leq N^{1/2} (\log{N})^{-100k}$, we get that the total number of these $p$ is $O(N (\log{N})^{-100k})$.

Now referring back to \eqref{eq:crudesum4-0}, we see that $\sideset{}{_2}\sum \ll F_{\max}^2  W^{-1/2} N (\log{N})^{-40k}$.
But this is $o(N)$, and so certainly also $o\left(\frac{\phi(W)^k}{W^{k+1}} N (\log{N})^{k}\right)$.
\end{proof}

\begin{proof}[Proof that $\sideset{}{_4}\sum = o\left(\frac{\phi(W)^k}{W^{k+1}} N (\log{N})^{k}\right)$] We proceed as above, but now with $\Qq$ equal to the union of the sets $\Pp_q$ for $q > N^{1/2} (\log{N})^{100k}$. We will show that $\#\Qq \cap [N,3N] \ll N (\log{N})^{-200k}$. By the previous Cauchy-ing argument, this is (more than) enough. If $p \in \Qq \cap [N,3N]$, then the order of $g$ modulo $p$, call it $\ell$, divides $(p-1)/q$ for some $q > N^{1/2} (\log{N})^{100k}$. In particular, $\ell < 3N^{1/2} (\log{N})^{-100k}$. Since $g^\ell-1$ has only $O(\ell)$ prime factors, summing on $\ell <  3N^{1/2} (\log{N})^{-100k}$ shows that there are $O(N (\log{N})^{-200k})$ possibilities for $p$.
\end{proof}

\subsubsection{Estimation of $\sideset{}{_3}\sum$} For each prime $q$, we let $\A_q$ denote the set of natural numbers $n \equiv 1\pmod{q}$. We estimate  $\sideset{}{_3}\sum$ using the trivial bound $\chi_{\Pp_q} \leq \chi_{\A_q}$. To save space, write $\I:= (N^{1/2} (\log{N})^{-100k}, N^{1/2} (\log{N})^{100k}]$. Then
\[ \sideset{}{_3}\sum \leq \sum_{q \in \I} \sum_{\substack{N \leq n < 2N \\ n \equiv \nu\pmod{W}}} \chi_{\A_q}(n+h_m) w(n).  \]
Expanding out the right-hand side yields
\begin{equation}\label{eq:threesums} \sum_{q \in \I} \sum_{\substack{d_1, \dots, d_k \\ e_1, \dots, e_k}} \lambda_{d_1, \dots, d_k} \lambda_{e_1, \dots, e_k} \sum_{\substack{N \leq n < 2N\\ n \equiv \nu\pmod{W} \\ [d_i, e_i] \mid n+h_i\,\forall i}} \chi_{\A_q}(n+h_m).\end{equation}
We can assume $d_1 \cdots d_k$ is a squarefree integer coprime to $W$ and not exceeding $R$, since otherwise $\lambda_{d_1, \dots, d_k}=0$. A similar assumption can be made for $e_1 \cdots e_k$. Since $q \in \I$, it follows that $q$ is coprime to each $d_i$, each $e_i$, and $W$. Now the innermost sum in \eqref{eq:threesums} vanishes unless $[d_1, e_1], [d_2, e_2], \dots, [d_k, e_k]$, and $W$ are pairwise coprime. Using a $'$ to denote this restriction on the $d_i$ and $e_i$, we get that
\begin{align*} \sum_{q \in \I} \sum_{\substack{d_1, \dots, d_k \\ e_1, \dots, e_k}} \lambda_{d_1, \dots, d_k} \lambda_{e_1, \dots, e_k} &\sum_{\substack{N \leq n < 2N\\ n \equiv \nu\pmod{W} \\ [d_i, e_i] \mid n+h_i\,\forall i}} \chi_{\A_q}(n+h_m)\\&= \sum_{q \in \I} \sideset{}{'}\sum_{\substack{d_1, \dots, d_k \\ e_1, \dots, e_k}} \lambda_{d_1, \dots, d_k} \lambda_{e_1, \dots, e_k} \left(\frac{N}{qW\prod_{i=1}^{k}[d_i, e_i]} + O(1)\right). \end{align*}
The error here is
\begin{align*} \ll \left(\sum_{q \in \I} 1\right) \left(\sum_{d_1, \dots, d_k} |\lambda_{d_1, \dots, d_k}|\right)^2 &\ll N^{1/2} (\log{N})^{100k} \cdot \lambda_{\max}^2 \left(\sum_{r < R}\mu^2(r) \tau_k(r)\right)^2. \end{align*}
Recalling that $\lambda_{\max} \ll F_{\max} (\log{R})^k$ and that $\sum_{r < R} \tau_k(r) \ll R (\log{R})^{k-1}$, our final $O$ error term is $O(F_{\max}^2 \cdot N^{1/2} R^2 \cdot (\log{N})^{104k})$. Since $R=N^{\theta}$ with $\theta < \frac{1}{4}$, this error is $o(N)$ and so is negligible for us. We now turn attention to the main term, which has the form
\[ \Bigg(\sum_{q\in \I} \frac{1}{q}\Bigg) \Bigg(\frac{N}{W}\sideset{}{'}\sum_{\substack{d_1, \dots, d_k \\ e_1, \dots, e_k}} \frac{\lambda_{d_1, \dots, d_k} \lambda_{e_1, \dots, e_k}}{\prod_{i=1}^{k}[d_i, e_i]}\Bigg).\]
The first factor here is $O(\frac{\log\log{N}}{\log{N}})$, and so in particular is $o(1)$. Maynard's analysis (see the proofs of \cite[Lemmas 5.1, 6.2]{maynard14}) shows that the second factor here satisfies the asymptotic formula asserted for $S_1$ in Proposition \ref{prop:main-maynard}. Hence, $\sideset{}{_3}\sum = o(\frac{\phi(W)^k}{W^{k+1}} N (\log{N})^{k})$, as desired.

\subsubsection{Estimation of $\sideset{}{_1}\sum$} For this case, let $\I:= (\log\log\log{N}, (\log{N})^{100k}]$. Using the bound $\chi_{\Pp_q} \leq \chi_{\Pp_q^{(0)}}$, we get that
\[ \sideset{}{_1}\sum \leq \sum_{q \in \I} \sum_{N \leq n < 2N} \chi_{\Pp^{(0)}_q}(n+h_m) w(n). \]
Expanding out the right-hand side gives
\begin{equation}\label{eq:expandsum2} \sum_{q \in \I} \sum_{\substack{d_1, \dots, d_k\\ e_1, \dots, e_k}} \lambda_{d_1, \dots, d_k} \lambda_{e_1, \dots, e_k} \sum_{\substack{N \leq n < 2N \\ n \equiv \nu\pmod{W} \\ [d_i, e_i] \mid n+h_i\,\forall i}} \chi_{\Pp^{(0)}_q}(n+h_m). \end{equation}
The inner sum can be written as a sum over a single residue class modulo $f:=W \prod_{i=1}^{k} [d_i, e_i]$, provided that $W, [d_1, e_1], \dots, [d_k, e_k]$ are pairwise coprime; otherwise we get no contribution.  We also need that $n+h_m$ lies in a residue class coprime to $f$, which happens precisely when $d_m=e_m=1$. Also, $\chi_{\Pp^{(0)}_q}(n+h_m)$ vanishes unless $q \mid n+h_m-1$, and this implies that the inner sum in \eqref{eq:expandsum2} vanishes unless $q$ is coprime to each $d_i$ and $e_i$. Indeed, if $q$ divides $d_i$ or $e_i$ without the inner sum vanishing, then $q \mid h_m-h_i-1$. But that divisibility cannot hold for $q \in \I$, since $0 < |h_m-h_i-1|< k \cdot K!$.

Thus, we only see a contribution to \eqref{eq:expandsum2} if $[d_1, e_1]$, $[d_2, e_2]$, \dots, $[d_k, e_k]$, $W$, and $q$ are pairwise coprime. Under these conditions, we claim that
\begin{multline}\label{eq:chebestimate} \sum_{\substack{N \leq n < 2N \\ n \equiv \nu\pmod{W} \\ [d_i, e_i] \mid n+h_i\,\forall i}} \chi_{\Pp^{(0)}_q}(n+h_m) \\ = \frac{1}{q(q-1) \phi(W) \prod_{i=1}^{k} \phi([d_i, e_i])} \int_{N+h_m}^{2N+h_m} \frac{dt}{\log{t}} + O(N^{1/2}\log{N}). \end{multline}
To see this, let $p:=n+h_m$. Then the prime $p \in [N+h_m, 2N+h_m)$ makes a contribution to the the left-hand sum precisely when $\Frob_{p}$ is a certain element of $\Q(\zeta_{f})$ --- determined by the congruence conditions modulo the $[d_i,e_i]$ and $W$ --- \emph{and} when $p$ splits completely in $\Q(\zeta_q, \sqrt[q]{g})$. Now $\sqrt[q]{g} \not \subset \Q(\zeta_{qf})$, since $\Q(\sqrt[q]{g})$ is not a Galois extension of $\Q$. Thus, letting $L:= \Q(\zeta_{qf}, \sqrt[q]{g})$, we find that \begin{align*} [L:\Q] &= [L: \Q(\zeta_{qf})] [\Q(\zeta_{qf}): \Q] \\&= q \cdot\phi(qf) = q(q-1) \phi(W) \prod_{i=1}^{k} \phi([d_i, e_i]). \end{align*}
Hence, $\Q(\zeta_{f})$ and $\Q(\zeta_q, \sqrt[q]{g})$ are linearly disjoint extensions of $\Q$ with compositum $L$. Our conditions on $p$ amount to placing $\Frob_p$ in a certain uniquely determined conjugacy class of size $1$ in $\Gal(L/\Q)$.
Since the only primes that ramify in $L$ divide $qfg$, Lemma \ref{lem:discbound} gives that
\[ \log |\Delta_L| \ll [L:\Q](\log{(qfg)} + \log[L:\Q]) \ll [L:\Q]\log{N}. \]
Inserting this estimate into Theorem \ref{thm:serre} now yields \eqref{eq:chebestimate}.

Returning now to \eqref{eq:expandsum2}, we see that the error term in  \eqref{eq:chebestimate} yields a total error of size
\begin{align*} \ll N^{1/2}\log{N} \left(\sum_{q \in \I} 1\right) \left(\sum_{d_1, \dots, d_k} |\lambda_{d_1, \dots, d_k}|\right)^2 &\ll N^{1/2} (\log{N})^{100k+1} \cdot \lambda_{\max}^2 \left(\sum_{r < R}\tau_k(r)\right)^2\\
&\ll F_{\max}^2 \cdot N^{1/2} R^2 \cdot (\log{N})^{104k+1}.\end{align*}
This is $o(N)$ and so is again negligible for us. Letting $X_N:= \int_{N+h_m}^{2N+h_m} dt/\log{t}$, the main term has the shape
\begin{equation}\label{eq:complicatedmainterm} \sum_{q \in \I} \frac{1}{q(q-1)}\left(\frac{X_N}{\phi(W)} \sideset{}{'}\sum_{\substack{d_1, \dots, d_k \\ e_1, \dots, e_k \\ d_m=e_m=1}} \frac{\lambda_{d_1, \dots, d_k} \lambda_{e_1, \dots, e_k}}{\prod_{i=1}^{k}\phi([d_i, e_i])} \right).\end{equation}
Here the $'$ on the sum indicates that $W, [d_1, e_1], \dots, [d_k, e_k]$, and $q$ are pairwise coprime. Owing to the support of the $\lambda$'s, this restriction on the sum has the same effect as requiring that $(d_i, e_j)=1$ for all $i \neq j$ and that $(d_i,q)=(e_j,q)=1$ for all $1\leq i, j \leq k$. We incorporate the restrictions that $(d_i, e_j)=1$ by multiplying through by $\sum_{s_{i,j} \mid d_i, e_j} \mu(s_{i,j})$ for $i \neq j$. Similarly, we incorporate the restrictions that $(d_i,q)=(e_j,q)=1$ by multiplying through by $\sum_{\delta_i \mid d_i, q} \mu(\delta_i)$ and $\sum_{\epsilon_j \mid e_j, q} \mu(\epsilon_j)$, for all pairs of $i$ and $j$. Let $g$ be the completely multiplicative function defined by $g(p)=p-2$ for all primes $p$, and note that
\[ \frac{1}{\phi([d_i,e_i])} = \frac{1}{\phi(d_i) \phi(e_i)}\sum_{u_i \mid d_i, e_i} g(u_i) \]
for squarefree $d_i$ and $e_i$. This allows us to rewrite the parenthesized portion of \eqref{eq:complicatedmainterm} as
\begin{multline}\label{eq:verycomplicated} \frac{X_N}{\phi(W)}\sum_{u_1, \dots, u_k} \left(\prod_{i=1}^{k} g(u_i)\right) \sideset{}{^*}\sum_{s_{1,2}, \dots, s_{k, k-1}} \left(\prod_{\substack{1 \leq i, j \leq k \\ i \neq j}} \mu(s_{i,j})\right) \sum_{\substack{\delta_1, \dots, \delta_k \mid q \\ \epsilon_1, \dots, \epsilon_k \mid q}} \left(\prod_{i=1}^{k} \mu(\delta_i)\prod_{j=1}^{k} \mu(\epsilon_j)\right) \\ \times \sum_{\substack{d_1,\dots, d_k \\ e_1, \dots, e_k \\ u_i \mid d_i, e_i\,\forall i \\ s_{i,j}\mid d_i, e_j\,\forall i\neq j\\ \delta_i \mid d_i, \epsilon_j \mid e_j\, \forall i, j \\ d_m=e_m=1}} \frac{\lambda_{d_1,\dots, d_k} \lambda_{e_1, \dots, e_k}}{\prod_{i=1}^{k}\phi(d_i)\phi(e_i)},\end{multline}
where the $*$ on the sum indicates that $s_{i,j}$ is restricted to be coprime to $u_i$, $u_j$, $s_{i,a}$, and $s_{b,j}$ for all $a \neq i$ and $b\neq j$. (The other values of $s_{i,j}$ make no contribution.) Introducing the new variables
\[ y_{r_1, \dots, r_k}^{(m)} :=  \left(\prod_{i=1}^{k} \mu(r_i) g(r_i)\right) \sum_{\substack{d_1, \dots, d_k \\ r_i \mid d_i\,\forall i\\ d_m=1}} \frac{\lambda_{d_1, \dots, d_k}}{\prod_{i=1}^{k} \phi(d_i)}, \]
we may rewrite \eqref{eq:verycomplicated} as
\begin{multline*} \frac{X_N}{\phi(W)} \sum_{u_1, \dots, u_k} \left(\prod_{i=1}^{k} g(u_i)\right) \sideset{}{^*}\sum_{s_{1,2}, \dots, s_{k, k-1}} \left(\prod_{\substack{1 \leq i, j \leq k \\ i \neq j}} \mu(s_{i,j})\right) \sum_{\substack{\delta_1, \dots, \delta_k \mid q \\ \epsilon_1, \dots, \epsilon_k \mid q}} \left(\prod_{i=1}^{k} \mu(\delta_i) \prod_{j=1}^{k}\mu(\epsilon_j)\right) \\ \times \Bigg(\prod_{i=1}^{k} \frac{\mu(a_i)}{g(a_i)}\Bigg) \Bigg(\prod_{j=1}^{k} \frac{\mu(b_j)}{g(b_j)}\Bigg) y_{a_1, \dots, a_k}^{(m)} y_{b_1, \dots, b_k}^{(m)}, \end{multline*}
where $a_i  = \lcm[u_i \prod_{j \neq i} s_{i,j}, \delta_i]$ and $b_j = \lcm[u_j \prod_{i \neq j} s_{i,j}, \epsilon_j]$. Define $\delta_i' \in \{1,q\}$ and $\epsilon_j' \in \{1, q\}$ by the equations
\[ a_i = \Bigg(u_i \prod_{j \neq i} s_{i,j}\Bigg) \delta_i', \qquad b_j = \Bigg(u_j \prod_{i \neq j} s_{i,j}\Bigg) \epsilon_j'. \]
Exploiting coprimality, we can write $\mu(a_i) = \left(\mu(u_i) \prod_{j \neq i} \mu(s_{i,j})\right) \mu(\delta_i')$, and similarly for $\mu(b_j)$, $g(a_i)$, and $g(b_j)$. This transforms \eqref{eq:verycomplicated} into
\begin{multline*} \frac{X_N}{\phi(W)} \sum_{u_1, \dots, u_k} \left(\prod_{i=1}^{k} \frac{\mu(u_i)^2}{g(u_i)} \right) \sideset{}{^*}\sum_{s_{1,2}, \dots, s_{k, k-1}}\left(\prod_{\substack{1 \leq i, j \leq k\\i \neq j}} \frac{\mu(s_{i,j})}{g(s_{i,j})^2}\right) \\ \times \sum_{\substack{\delta_1, \dots, \delta_k \mid q \\ \epsilon_1, \dots, \epsilon_k \mid q}}\left(\prod_{i=1}^{k} \frac{\mu(\delta_i)\mu(\delta_i')}{g(\delta_i')} \prod_{j=1}^{k}\frac{\mu(\epsilon_j) \mu(\epsilon_j')}{g(\epsilon_j')}\right) y_{a_1, \dots, a_k}^{(m)} y_{b_1, \dots, b_k}^{(m)}. \end{multline*} Let $y_{\max}^{(m)} = \max_{r_1, \dots, r_k} |y_{r_1, \dots, r_k}^{(m)}|$.  From \cite[eq. (6.10)]{maynard14}, we have $y_{\max}^{(m)} \ll F_{\max} \frac{\phi(W)}{W} \log{R}$. Inserting these bounds into the previous display, we find that  \eqref{eq:verycomplicated} is
\begin{multline*} \ll \frac{X_N}{\phi(W)}\bigg(\sum_{\substack{u < R \\ \gcd(u,W)=1}}\frac{\mu(u)^2}{g(u)}\bigg)^{k-1} \left(\sum_{s} \frac{\mu(s)^2}{g(s)^2}\right)^{k(k-1)} {y_{\max}^{(m)}}^2 \\ \ll F_{\max}^2 \cdot \frac{X_N}{\phi(W)} \left(\frac{\phi(W)}{W}\right)^{k+1} (\log{R})^{k+1} \ll F_{\max}^2 \left(\frac{\phi(W)^k}{W^{k+1}}\right) N (\log{N})^k. \end{multline*}
We used here that there are only $O(1)$ possibilities for the $\delta_i$ and $\epsilon_j$, and that for each of these, $\prod_{i}\frac{1}{g(\delta_i')}  \prod_{j} \frac{1}{g(\epsilon_j')} \le 1$. Referring back to \eqref{eq:complicatedmainterm}, we see that our original main term contributes
\[ \ll F_{\max}^2 \left(\frac{\phi(W)^k}{W^{k+1}}\right) N (\log{N})^k \sum_{q \in \I} \frac{1}{q(q-1)} = o\left(\frac{\phi(W)^k}{W^{k+1}} N (\log{N})^k\right), \]
as desired.

\begin{remark}
The truth of Theorem \ref{thm:main} could also have been predicted on heuristic grounds. Indeed, there are well known heuristics for Artin's primitive root conjecture, suggesting even the `correct' value of $c_g$ (see \cite[\S\S2--5]{moree12}), as well as heuristics for the prime $k$-tuples conjecture (see for instance,
\cite[pp. 14--15]{CP05}), and these can be fitted together. As an example, this combined heuristic suggests that the count of twin prime pairs $p, p+2$ with $p \leq x$ and with $2$ a primitive root of both $p$ and $p+2$ should be approximately
\[ \mathfrak{S} \int_{2}^{x} \frac{dt}{(\log{t})^2}, \quad\text{where}\quad \mathfrak{S} := \frac{1}{4} \prod_{p > 3} \left(1-\frac{3}{(p-1)^2}\right). \]
Quantitative conjectures of this kind, but in the context of primes represented by a single irreducible polynomial rather than primes produced by linear forms, appear in recent work of Moree \cite{moree07} and of Akbary and Scholten \cite{AS14}.
\end{remark}

\section{Concluding remarks}
We conclude with a proof of the following result, which seems of independent interest:

\begin{thm}[conditional on GRH]\label{thm:consecutive} Fix an integer $g\neq -1$ and not a square. For every positive integer $m$, there are $m$ consecutive primes all of which possess $g$ as a primitive root.
\end{thm}

Theorem \ref{thm:consecutive} might be compared with Shiu's celebrated result \cite{shiu00} that each coprime residue class $a\bmod{q}$ contains arbitrarily long runs of consecutive primes. Our proof of Theorem \ref{thm:consecutive} is similar in spirit to a short proof of Shiu's theorem recently given by Banks, Freiberg, and Turnage-Butterbaugh \cite{BFTB14}. 

It will be useful to first translate the proof of Theorem \ref{thm:main} into probabilistic terms. Let $k$ be a fixed positive integer, and let $h_1, \dots, h_k$ be given by \eqref{eq:hidef}. We view the set of $n \in [N,2N)$ with $n \equiv \nu\pmod{W}$ as a finite probability space where the probability mass at each $n_0$ is given by
\[ w(n_0)/\sum_{\substack{N \leq n < 2N \\ n \equiv \nu\pmod{W}}} w(n). \]
Here the weights $w(n)$ are assumed to be of the form specified in Proposition \ref{prop:main-maynard}. Introduce the random variables
\[ X := \sum_{i=1}^{k} \chi_{\Pp}(n+h_i)\quad\text{and}\quad Y:=\sum_{i=1}^{k} \chi_{\Pp\setminus \tilde{\Pp}}(n+h_i). \]
Then $\E[X]= S_2/S_1$. Given suitable parameters $F$ and $\theta$, Proposition \ref{prop:main-maynard} gives us the limiting value of $\E[X]$ as $N\to\infty$. Combining Propositions \ref{prop:main-maynard} and \ref{prop:mklower}, we see that for $k$ large enough in terms of $m$, we can choose parameters so this limiting value exceeds $m-1$. On the other hand, it was shown in \S\ref{sec:proof} that (with the same choice of parameters) $\E[Y] = o(1)$ as $N\to\infty$. Thus, $\E[X-Y] > m-1$ for all large $N$. But $X-Y = \sum_{i=1}^{m} \chi_{\tilde{\Pp}}(n+h_i)$. Hence, for some $n \in [N,2N)$, the list $n+h_1, \dots, n+h_k$ contains at least $m$ primes having $g$ as a primitive root. Theorem \ref{thm:main} follows, with $C_m = h_k - h_1$.
 
We now present the minor variation of this argument needed to establish Theorem \ref{thm:consecutive}.

\begin{proof}[Proof of Theorem \ref{thm:consecutive}] Given $m$, we fix a large enough value of $k$ (and parameters $F, \theta$) so that the limiting value of $\E[X]$ exceeds $m-1$. Then for all large $N$,  \[ \Prob(X \geq m) \geq \E\left[\frac{X-(m-1)}{k}\right] = \frac{1}{k} (\E[X]-(m-1)) \gg 1. \] Note that $\Prob(Y > 0) \leq \E[Y]= o(1)$, as $N\to\infty$. So for large $N$, there is a positive probability that both $X \geq m$ and $Y=0$. This allows us to select $n \in [N,2N)$ with $n \equiv \nu\pmod{W}$ satisfying
\begin{enumerate}
\item[(i)] at least $m$ of $n+h_1, \dots, n+h_k$ are prime,
\item[(ii)] all of the primes among $n+h_1, \dots, n+h_k$ possess $g$ as a primitive root.
\end{enumerate}
We will argue momentarily that we can also assume
\begin{itemize}
\item[(iii)] the only primes in the interval $[n+h_1, n+h_k]$ are the primes in the list $n+h_1, \dots, n+h_k$.
\end{itemize}	
From (i), (ii), and (iii), we see that the set of primes in $[n+h_1, n+h_k]$ contains at least $m$ elements, all of which have $g$ as a primitive root. Theorem \ref{thm:consecutive} follows. 

In order to show we may assume (iii), we tweak the choice of the residue class $\nu \bmod{W}$ from which $n$ is sampled. In the proof of Lemma \ref{lem:presieving}, we chose $\nu_1$ as any odd integer avoiding $-h_1, \dots, -h_k$, $1-h_1, \dots, 1-h_k$ modulo $p$, for all odd $p\leq\log\log\log{N}$ not dividing $K$. We now add an extra condition on $\nu_1$. Choose distinct primes $p^{(h)} \in [\frac{1}{2}\log\log\log{N},\log\log\log N)$ for all even $h \in [h_1,h_k]\setminus \Hh$. We add the requirement that $\nu_1 \equiv -h \pmod{p^{(h)}}$ for each such $h$. This is consistent with our earlier restrictions, since $h$ is not congruent modulo $p^{(h)}$ to any of $h_1, \dots, h_k$ (since $h \not\in\Hh$) or to any of $h_1-1, \dots, h_k-1$ (since $h$ and the $h_i$ are all even). Using the resulting value of $\nu$ from Lemma \ref{lem:presieving}, we see that for even $h \in [h_1,h_k]\setminus\Hh$, we have $p_h \mid n+h$ whenever $n\equiv \nu\pmod{W}$. For all odd $h \in [h_1,h_k]$, we have trivially that $2 \mid n+h$ whenever $n \equiv \nu\pmod{W}$. Thus, $n+h$ is composite if $h \in [h_1,h_k] \setminus \Hh$, and so (iii) holds.\end{proof}

\providecommand{\bysame}{\leavevmode\hbox to3em{\hrulefill}\thinspace}
\providecommand{\MR}{\relax\ifhmode\unskip\space\fi MR }
\providecommand{\MRhref}[2]{%
  \href{http://www.ams.org/mathscinet-getitem?mr=#1}{#2}
}
\providecommand{\href}[2]{#2}


\begin{thebibliography}{10}

\bibitem{AS14}
A.~Akbary and K.~Scholten, \emph{Artin prime producing polynomials}, Math.
  Comp., to appear. Preprint version available online as
  \texttt{arXiv:1310.5198 [math.NT]}.

\bibitem{BFTB14}
W.~D. Banks, T.~Freiberg, and C.~L. Turnage-Butterbaugh, \emph{Consecutive
  primes in tuples}, submitted. Preprint version available online as
  \texttt{arXiv:1311.7003 [math.NT]}.

\bibitem{CP05}
R.~Crandall and C.~Pomerance, \emph{Prime numbers: a computational
  perspective}, second ed., Springer, New York, 2005.

\bibitem{GM84}
R.~Gupta and M.~R. Murty, \emph{A remark on {A}rtin's conjecture}, Invent.
  Math. \textbf{78} (1984), 127--130.

\bibitem{HB86}
D.~R. Heath-Brown, \emph{Artin's conjecture for primitive roots}, Quart. J.
  Math. Oxford Ser. (2) \textbf{37} (1986), 27--38.

\bibitem{hooley67}
C.~Hooley, \emph{On {A}rtin's conjecture}, J. Reine Angew. Math. \textbf{225}
  (1967), 209--220.

\bibitem{hooley76}
\bysame, \emph{Applications of sieve methods to the theory of numbers},
  Cambridge Tracts in Mathematics, no.~70, Cambridge University Press,
  Cambridge-New York-Melbourne, 1976.

\bibitem{LO77}
J.~C. Lagarias and A.~M. Odlyzko, \emph{Effective versions of the {C}hebotarev
  density theorem}, Algebraic number fields: {$L$}-functions and {G}alois
  properties ({P}roc. {S}ympos., {U}niv. {D}urham, {D}urham, 1975), Academic
  Press, London, 1977, pp.~409--464.

\bibitem{LP14}
H.~Li and H.~Pan, \emph{Bounded gaps between primes of the special form},
  preprint available as \texttt{arXiv:1403.4527 [math.NT]}, 2014.

\bibitem{maynard14}
J.~Maynard, \emph{Small gaps between primes}, Ann. Math., to appear. All
  references are to the preprint version available online as
  \texttt{arXiv:1311.4600v2 [math.NT]}.

\bibitem{moree07}
P.~Moree, \emph{Artin prime producing quadratics}, Abh. Math. Sem. Univ.
  Hamburg \textbf{77} (2007), 109--127.

\bibitem{moree12}
\bysame, \emph{Artin's primitive root conjecture---a survey}, Integers
  \textbf{12} (2012), 1305--1416.

\bibitem{serre81}
J.-P. Serre, \emph{Quelques applications du th\'eor\`eme de densit\'e de
  {C}hebotarev}, Inst. Hautes \'Etudes Sci. Publ. Math. \textbf{54} (1981),
  323--401.

\bibitem{shiu00}
D.~K.~L. Shiu, \emph{Strings of congruent primes}, J. London Math. Soc. (2)
  \textbf{61} (2000), 359--373.

\bibitem{thorner14}
J.~Thorner, \emph{Bounded gaps between primes in {C}hebotarev sets}, preprint
  available as \texttt{arXiv:1401.6677 [math.NT]}, 2014.

\bibitem{wan97}
D.~Wan, \emph{Generators and irreducible polynomials over finite fields}, Math.
  Comp. \textbf{66} (1997), 1195--1212.

\end{thebibliography}
\end{document}